\documentclass[1p]{elsarticle}

\usepackage{amsfonts,amsmath,amssymb,amsthm,bbm}
\usepackage{mathrsfs,mathtools,stmaryrd,wasysym}
\usepackage{xspace,mydef}
\usepackage{esint}
\usepackage{xcolor}
\usepackage{hyperref,cleveref}
\DeclareMathAlphabet{\mathpzc}{OT1}{pzc}{m}{it}

\newcommand{\TheTitle}{Approximation of higher-order powers of the spectral fractional Laplacian via polyharmonic extension}

\theoremstyle{definition}

\newtheorem{remark}{Remark}

\theoremstyle{plain}
\newtheorem{theorem}{Theorem}
\newtheorem{proposition}{Proposition}

\begin{document}

\title{\TheTitle}

\author[1]{Enrique Ot\'arola}
\ead{enrique.otarola@usm.cl}
\affiliation[1]{
  organization = {Departamento de Matem\'atica, Universidad T\'ecnica Federico Santa Mar\'ia},
  city = {Valpara\'iso},
  country = {Chile}
}

\author[2]{Abner J.~Salgado}
\ead{asalgad1@utk.edu}
\affiliation[2]{
  organization = {Department of Mathematics, University of Tennessee},
  city ={Knoxville, TN},
  postcode = {37996},
  country={USA}
}


\begin{abstract}
We use the polyharmonic extension approach to develop a numerical technique for discretizing higher-order powers of the spectral fractional Laplacian $\Laps$ with $s \in (1,2)$.
\end{abstract}

\maketitle

\section{Introduction}
\label{sec:Into}

In their seminal 2007 work \cite{MR2354493}, Caffarelli and Silvestre introduced the harmonic extension approach to localize, in $\Real^d$, the fractional Laplacian of order between zero and one. While these ideas were to some extent already known \cite{MR247668}, it is difficult to overstate the impact of \cite{MR2354493}.  Any attempt to provide a comprehensive list of extensions and applications would inevitably overlook important contributions. Therefore, we restrict the discussion to noting that this work opened the door to a PDE-based approach for approximating solutions to certain fractional diffusion problems; see \cite{MR3348172,MR3989717}.

The original work \cite{MR2354493} and the numerical developments \cite{MR3348172,MR3989717} were limited to fractional powers of order between zero and one. One variant of the harmonic extension approach is the use of \emph{polyharmonic extensions} to study higher-order fractional powers of the Laplacian, that is, $\Laps$ for $s>1$. This was initiated in \cite{1302.4413}\footnote{see, however, the comments in \cite{MR4429579} about this work.} and further developed in \cite{MR3773814,MR4429579,MusinaNazarov,MR4742773}.

In this work, we inaugurate the numerical analysis of problems involving higher-order fractional Laplacians. We focus on
$s \in (1,2)$ and provide a numerical counterpart to \cite{MR3773814,MR4429579,MusinaNazarov,MR4742773}. Specifically, we develop a PDE-based approach for the numerical resolution of the following problem. Assume $d \in \Natural$, $s \in (1,2)$, and let $\Omega \subset \Real^d$ be a convex polytope. Given $f : \Omega \to \Real$, find $u : \bar\Omega \to \Real$ such that
\begin{equation}
\label{eq:TheProblem}
  \Laps u =f \text{ in } \Omega.
\end{equation}
Here, $\Laps$ denotes the \emph{spectral} fractional Dirichlet Laplacian.

In our presentation, we use the context and notation from \cite{MR3348172,MR3989717} and draw on several results from these works.

\vspace{-0.3cm}

\section{Polyharmonic extensions}
\label{sec:MusinaNazarov}

Let $\calH$ be a separable Hilbert space, and let $\calL : \mathcal{D}(\calL) \to \calH$ be a linear, unbounded, positive, self-adjoint operator with discrete spectrum; that is,
there exists $\{(\lambda_k,\varphi_k)\}_{k\in \Natural}$ $\subset \Real_+ \times \calH$ such that $\{\varphi_k\}_{k\in \Natural}$ is an orthonormal basis of $\calH$ and
$
  \calL \varphi_k = \lambda_k \varphi_k
$
for every $k \in \Natural$. For $s \in \Real$, we define the $s$-th power of $\calL$ in the sense of \emph{spectral theory} as
\[
  \calL^s w = \sum_{k =1}^\infty \lambda_k^s W_k \varphi_k,
  \qquad
  W_k = (w,\varphi_k)_\calH,
  \qquad
  k \in \Natural.
\]
For $s > 0$, we define the Hilbert space
\[
  \calH^s_\calL \coloneqq \calD(\calL^{s/2}) = \left\{ w = \sum_{k=1}^\infty W_k \varphi_k \in \calH \ : \ \| w \|_{\calH^s_\calL}^2
  \coloneqq
  (w,w)_{\calH^s_\calL}
  =
  \sum_{k=1}^\infty \lambda_k^s |W_k|^2 < \infty \right\},
\]
where $(v,w)_{\calH^s_\calL} \coloneqq (\calL^{s/2} v, \calL^{s/2} w)_{\calH}$ for $v,w \in \calH^s_\calL$. We identify $\left(\calH^{s}_\calL \right)'$ with $\calH^{-s}_\calL = \{ \calL^s v: v \in \calH^s_\calL \}$ using the identity $\langle \calL^s v,w \rangle =  (\calL^{s/2} v, \calL^{s/2} w)_{\calH}$ for $v,w \in \calH^s_\calL$. Therefore, we can write elements $f \in \calH^{-s}_\calL$ as follows:
\[
 f = \sum_{k=1}^\infty F_k \varphi_k:
\qquad
\| f \|_{\calH^{-s}_\calL} \coloneqq \left( \sum_{k=1}^\infty \lambda_k^{-s} |F_k|^2 \right)^{\frac12} < \infty.
\]
This extends the definition of the norm $\| \cdot \|_{\calH^s_\calL}$ to $s < 0$. Note that $\calL^s : \calH^s_\calL \to \calH^{-s}_\calL$ is an isometry with inverse $\calL^{-s}$.

We now follow \cite[Theorem 1.2]{MusinaNazarov} to describe the polyharmonic extension approach for $s \in (1,2)$. Let
\[
  \frakb \coloneqq 3 -2s \in (-1,1).
\]
We introduce the weighted space $L^2(y^\frakb, \Real_+; \calH)$, with $y \mapsto |y|^\frakb \in A_2(\Real)$, and define
\begin{align*}
 H^1_{\calL}(y^\frakb, \Real_+; \calH) \coloneqq H^1(y^\frakb, \Real_+; \calH) \cap L^2(y^\frakb, \Real_+; \calH_{\calL}),
 \\
 H^2_{\calL}(y^\frakb, \Real_+; \calH) \coloneqq H^2(y^\frakb, \Real_+; \calH) \cap L^2(y^\frakb, \Real_+; \calH_{\calL}^2).
\end{align*}
On the weighted space $L^2(y^\frakb, \Real_+; \calH)$, we define the (unbounded) operators
\[
  \polD_\frakb w \coloneqq -\frac1{y^\frakb} \diff_y \left( y^\frakb \diff_y w \right) = - \diff_{yy}^2 w - \frac{\frakb}{y} \diff_y w,
  \qquad
  \polL_\frakb w \coloneqq \polD_\frakb w + \calL w.
\]

\begin{theorem}[extension]
\label{thm:MusinaNazarov}
  Let $s \in (1,2)$ and $f \in \calH^{-s}_\calL$. If $\fraku \in H^2_{\calL}(y^\frakb, \Real_+; \calH)$ solves
  \begin{equation}
  \label{eq:ExtensionH}
    \polL_\frakb^2 \fraku = 0 \text{ in } \Real_+ \times \calH,
    \qquad
    \lim_{y\downarrow 0} y^{\frakb} \diff_y \fraku = 0 \text{ in } \calH,
    \qquad
    -\lim_{y\downarrow 0} y^\frakb \diff_y \polL_\frakb \fraku = d_s f \text{ in } \calH_{\calL}^{-s},
  \end{equation}
  then $u \coloneqq \fraku(0) \in \calH^s_\calL$ satisfies $\calL^s u = f$. Here, $d_s \coloneqq 2^\frakb \frac{\Gamma(2-s)}{\Gamma(s)}$. Moreover, if
  \begin{equation}
  \label{eq:PsiAndPsik}
    c_s \coloneqq \frac{2^{1-s}}{\Gamma(s)}, \qquad \qquad \psi(z) \coloneqq c_s z^s K_s(z), \qquad \psi_k(y) \coloneqq \psi( \sqrt{\lambda_k} y ),
  \end{equation}
  where $K_s$ denotes the modified Bessel function of the second kind of order $s$, then
  \begin{equation}
  \label{eq:SolutionRepresentation}
    \fraku(y) = \sum_{k=1}^\infty U_k \psi_k(y) \varphi_k, \qquad U_k = (u,\varphi_k)_\calH.
  \end{equation}
\end{theorem}

\subsection{Higher powers of the spectral fractional Laplacian}

We now specialize \Cref{thm:MusinaNazarov} to the spectral fractional Laplacian. Let $d \in \Natural$ and $\Omega \subset \Real^d$ be a bounded convex polytope. Set $\calH = \Ldeux$ and $\calL = -\LAP_x$, the Dirichlet Laplacian. Then, $\calH^s_\calL = \polH^s(\Omega)$ and, recalling that $\calC \coloneqq \Omega \times (0,\infty)$,
\[
  L^2(y^\frakb,\Real_+;\calH) = L^2(y^\frakb,\Real_+;\Ldeux) = L^2(y^\frakb, \calC ),
  \quad
  \| w \|_{L^2(y^\frakb,\calC)}^2 \coloneqq \int_0^\infty y^\frakb \int_\Omega |w|^2 \diff x \diff y.
\]
The spaces $\calH_{\calL}$ and $\calH_{\calL}^2$ are as follows: $\calH_{\calL} = H_0^1(\Omega)$ and $\calH_{\calL}^2 = H^2(\Omega) \cap H_0^1(\Omega)$.
Define $\partial_L \calC \coloneqq \partial\Omega \times (0,\infty)$. Then, \cite[Section 5.3]{MusinaNazarov}
\begin{align*}
 H^1_{\calL}(y^\frakb, \Real_+; \calH) & =  \left\{ w \in H^1(y^\frakb, \calC) \ : \:  \ w_{|\partial_L\calC} = 0 \right\},
 \\
 H^2_{\calL}(y^\frakb, \Real_+; \calH) & = \left\{ w \in H^1_{\calL}(y^\frakb, \Real_+; \calH) \ : \:  y^{\frakb} \nabla w \in H^1(y^{-\frakb},\calC) \right\}.
\end{align*}
Note that $\polL_\frakb w = \polD_\frakb w + \calL w = - \diff_{yy}^2 w - \frakb y^{-1} \diff_y w - \Delta_{x} w = - y^{-\frakb}\mathrm{div}(y^{\frakb} \nabla w)$. Define
\[
  \HtwoLb(y^\frakb,\calC) \coloneqq \left\{ w \in H^2(y^\frakb, \calC) \ : \: \lim_{y \downarrow0} y^{\frakb}\partial_y w = 0, \ w_{|\partial_L\calC} = \partial_\bfn w_{|\partial_L\calC} = 0 \right\}.
\]
We endow this space with the norm \cite[Section 5.3]{MusinaNazarov}
\[
  \| w \|_{\HtwoLb(y^\frakb,\calC)} \coloneqq \left\| \polL_\frakb w \right\|_{L^2(y^\frakb,\calC)} = \left( \int_0^{\infty}\int_{\Omega} y^{-\frakb} |\mathrm{div}(y^{\frakb} \nabla w)|^2 \mathrm{d}x \mathrm{d}y\right)^{\frac12}.
\]

A weak formulation for \eqref{eq:ExtensionH} is as follows: find $\fraku \in \HtwoLb(y^\frakb,\calC)$ such that
\begin{equation}
\label{eq:ExtensionWeakForm}
  \int_\calC y^\frakb \polL_\frakb \fraku \polL_\frakb v \diff x \diff y = d_s \langle f, \Tr_\Omega v\rangle_{\calH_{\calL}^{-s},\calH_{\calL}^s} \quad \forall v \in \HtwoLb(y^\frakb, \calC).
\end{equation}
\Cref{thm:MusinaNazarov} then shows that $u \coloneqq \Tr_\Omega \fraku = (-\Delta)^{-s} f$.

\section{Regularity and truncation}

We now justify truncating $\calC$ to $\calC_\calY \coloneqq \Omega \times (0,\calY)$ for $\calY>0$. This relies on \eqref{eq:SolutionRepresentation}, which coincides with \cite[(2.24)]{MR3348172} and \cite[(4.1)]{MR3989717}, taking into account that now $s \in(1,2)$.

\begin{proposition}[exponential decay]
\label{prop:ExpDecay}
  For every $a, b > 0$, $a<b$, the function \eqref{eq:SolutionRepresentation} satisfies
  \begin{equation}
  \label{eq:ExtensionYonInterval}
    \| \polL_\frakb \fraku \|_{L^2(y^\frakb, \Omega \times (a, b))} \leq \sup_{k =1}^\infty I_k(a,b) \| f \|_{\polH^{-s}(\Omega)},
    \quad
    I_k(a,b)^2 \coloneqq 4 c_s^2 \int_{\sqrt{\lambda_k}a}^{\sqrt{\lambda_k}b} z K_{s-1}^2(z) \diff z.
  \end{equation}
  In particular, for every $\calY \geq \tfrac1{\sqrt{\lambda_1}}$, we have
  \begin{equation}
  \label{eq:ExtensioYtruncation}
    \| \polL_\frakb \fraku \|_{L^2(y^\frakb, \Omega \times (\calY, \infty))} \lesssim \exp\left( -  \frac{ \sqrt{\lambda_1} \calY}2 \right) \| f \|_{\polH^{-s}(\Omega)}.
  \end{equation}
\end{proposition}
\begin{proof}
  Owing to \eqref{eq:SolutionRepresentation} and the orthogonality properties of $\{\varphi_k\}_{k=1}^\infty$, we have
  \begin{align*}
    \int_a^b y^\frakb \int_\Omega |\polL_\frakb \fraku|^2 \diff x \diff y &=
    \int_a^b y^\frakb \int_\Omega \left| \sum_{k=1}^\infty U_k \left[ \polD_\frakb \psi_k(y) + \lambda_k \psi_k(y) \right] \varphi_k(x) \right|^2 \diff x \diff y
    \\
    &= \sum_{k=1}^\infty |U_k|^2 \int_a^b y^\frakb \left| (\polD_\frakb + \lambda_k) \psi_k \right|^2 \diff y
      = \sum_{k=1}^\infty \lambda_k^s|U_k|^2 J_k(a,b)^2,
  \end{align*}
  where $J_k(a,b)^2 \coloneqq \lambda_k^{-s} \int_{a}^{b} y^\frakb | (\polD_\frakb + \lambda_k) \psi_k|^2 \diff y$. We now employ the change of variable $z = \sqrt{\lambda_k}y$, formula \cite[(3.2)]{MusinaNazarov}, denote by $\tilde{\polD}_\frakb$ the same operator as $\polD_\frakb$ but with respect to the variable $z$, use \cite[(10.29.1)]{NIST:DLMF}, or alternatively \cite[(3.3)]{MusinaNazarov}, and finally recall that $\psi$ is defined in \eqref{eq:PsiAndPsik} to obtain
  \begin{equation}
  \label{eq:JkIk}
    J_k(a,b)^2 =  \int_{\sqrt{\lambda_k}a}^{\sqrt{\lambda_k}b} z^\frakb \left| (\tilde{\polD}_\frakb + 1) \psi \right|^2 \diff z = 4 c_s^2 \int_{\sqrt{\lambda_k}a}^{\sqrt{\lambda_k}b} z K_{s-1}^2(z) \diff z = I_k(a,b)^2.
  \end{equation}
  Finally, if we set $a = \calY$, $b = \infty$, and use the exponential decay of $K_s$ \cite[\S2.5]{MR3348172}, we obtain
  \[
    I_k(\calY,\infty)^2 \leq I_1( \calY ,\infty)^2 \leq 4c_s^2 \exp(-\sqrt{\lambda_1}\calY) \int_0^\infty z^{\max\{ 3-2s, 0 \} } \exp(-z) \diff z
    \lesssim \exp(-\sqrt{\lambda_1}\calY),
  \]
  which implies \eqref{eq:ExtensioYtruncation}.
\end{proof}


Define
\[
  \HtwoLbY(y^\frakb,\calC_\calY) \coloneqq \left\{ w \in H^2(y^\frakb, \calC_\calY) \ : \
    \begin{array}{cc}
      \lim_{y \downarrow0} y^{\frakb}\partial_y w = 0, &
      w_{|\partial_L\calC} = \partial_\bfn w_{|\partial_L\calC} = 0
      \\ w_{|y=\calY} = \partial_y w|_{y =\calY} = 0 &
    \end{array}
  \right\}.
\]
If $\fraku_\calY \in \HtwoLbY(y^\frakb,\calC_\calY)$ solves the problem
\begin{equation}
\label{eq:ExtensionWeakFormY}
  \int_{\calC_\calY} y^\frakb \polL_\frakb \fraku_\calY \polL_\frakb v \diff x \diff y = d_s \langle f, \Tr_\Omega v\rangle_{\calH_{\calL}^{-s},\calH_{\calL}^s} \quad \forall v \in \HtwoLbY(y^\frakb, \calC_\calY),
\end{equation}
then $u_\calY \coloneqq \Tr_\Omega \fraku_\calY$ is an approximation of $u$, the soluton to \eqref{eq:TheProblem}, in the following sense.

\begin{proposition}[exponential approximation]
  Let $\calY \geq \tfrac1{\sqrt{\lambda_1}}$. If $\fraku$ and $\fraku_\calY$ solve \eqref{eq:ExtensionWeakForm} and \eqref{eq:ExtensionWeakFormY}, respectively, then we have
  \[
    \| u - u_\calY \|_{\polH^s(\Omega)} = \| \Tr_\Omega \fraku - \Tr_\Omega \fraku_\calY \|_{\polH^s(\Omega)}
    \lesssim \| \fraku - \fraku_\calY \|_{H^2(y^\frakb,\calC)} \lesssim \exp\left( -\frac{\sqrt{\lambda_1}\calY}4  \right) \| f \|_{\polH^{-s}(\Omega)}.
  \]
\end{proposition}
\begin{proof}
The trace estimate follows from \cite[(1.5)]{MusinaNazarov}, and the second estimate follows from \Cref{prop:ExpDecay} and an adaptation of \cite[Lemma 3.3]{MR3348172} to compare \eqref{eq:ExtensionWeakForm} and \eqref{eq:ExtensionWeakFormY}.
\end{proof}

\section{Discretization}

The finite element discretization of \eqref{eq:ExtensionWeakFormY} is as follows. Fix $\calY \geq \lambda_1^{-1/2}$, let $\{\Triang\}_{h>0}$ be a quasiuniform family of conforming triangulations of $\Omega$, and let $\{\calM_M\}_{M \in \polN}$ be a, possibly graded, family of partitions of $[0,\calY]$ with $M = \# \calM_M$. In this setting, we define
\begin{align*}
  \calS(\calM_M) &\coloneqq \left\{ p \in C^1([0,\calY]) \ : \ p_{|I_m} \in \polP_3, \ m =1,\ldots,M, \  p(\calY) = p'(0) = p'(\calY) = 0 \right\},
  \\
  W(\Triang) &\coloneqq \left\{ v_h \in C^1(\bar\Omega) \ : v_{h|T} \in \calP \, \ \forall T \in \Triang, \ v_{h|\partial\Omega} = \partial_\bfn v_{h|\partial\Omega} = 0 \right\},
  \\
  V_{\calY,h,M} &\coloneqq  W(\Triang) \otimes \calS(\calM_M),
\end{align*}
where $\calP$ is any $H^2$--conforming finite element space, for example, the Hermite element \cite[Example 3.2.6]{MR2373954}, the Argyris element \cite[Example 3.2.10]{MR2373954}, the composite HCT element \cite[Chapter VII, Section 46]{MR1115237} \cite{MR1266026}, etc. \cite{MR3163246}. Since $\frakb \in (-1,1)$ we have, for $p \in \calS(\calM_M)$,
$ y^\frakb p'(y) \to 0$ as $ y \downarrow 0$. Thus, $V_{\calY,h,M} \subset \HtwoLbY(y^\frakb,\calC_\calY)$. Let $\fraku_{\calY,h,M} \in V_{\calY,h,M}$ be the solution to
\begin{equation}
\label{eq:DiscretehM}
  \int_{\calC_\calY} y^\frakb \polL_\frakb \fraku_{\calY,h,M} \polL_\frakb v_{h,M} \diff x \diff y = d_s \langle f, \Tr_\Omega v_{h,M}\rangle, \qquad \forall v_{h,M} \in V_{\calY,h,M}.
\end{equation}
Our numerical approximation is then $u_{\calY,h,M} \coloneqq \Tr_\Omega \fraku_{\calY,h,M}$.Conformity implies that a C\'ea-type result is immediate. In addition, we can use the Cartesian product structure of $\calC_\calY$ as in \cite[Lemma 7]{MR3989717} to obtain the next result.

\begin{theorem}[best approximation]
\label{thm:DirectionalCea}
  Let $\fraku_\calY$ solve \eqref{eq:ExtensionWeakFormY} and $\fraku_{\calY,h,M}$ solve \eqref{eq:DiscretehM}, respectively. Let $\Pi_x : H^2_0(\Omega) \to W(\Triang)$ denote a linear projection that is stable in the sense that
  \[
    \| \Pi_x w \|_{L^2(\Omega)} \lesssim \| w \|_{L^2(\Omega)}, 
    \qquad
    \| \LAP \Pi_x w \|_{L^2(\Omega)} \lesssim \| \LAP w \|_{L^2(\Omega)}
    \quad \forall w \in H^2_0(\Omega).
  \]
  Let also $\Pi_y : H^2(y^\frakb,(0,\calY)) \to \calS(\calM_M)$ be a linear projection. Then,
  \begin{align*}
    \left\| \Tr_\Omega\left( \fraku_\calY - \fraku_{\calY,h,M} \right) \right\|_{\polH^s(\Omega)}
    &\lesssim \left\| \fraku_\calY - \fraku_{\calY,h,M} \right\|_{\HtwoLb(y^\frakb,\calC_\calY)}
    \\
    &= \min_{v_{h,M} \in V_{\calY,h,M}} \left\| \fraku_\calY - v_{h,M} \right\|_{\HtwoLb(y^\frakb,\calC_\calY)}
    \\
    &\lesssim
    \| (\id-\Pi_{x}) \polD_\frakb \fraku_\calY \|_{L^2(y^\frakb,\calC_\calY)}
    + \| \LAP_x(\id-\Pi_x) \fraku_\calY \|_{L^2(y^\frakb,\calC_\calY)}
    \\
    &+ \| \polD_\frakb (\id-\Pi_y) \fraku_\calY \|_{L^2(y^\frakb,\calC_\calY)}
    + \|  (\id-\Pi_y) \LAP_x\fraku_\calY\|_{L^2(y^\frakb,\calC_\calY)}
  \end{align*}
\end{theorem}
\begin{proof}
  We first use a trace estimate. The minimum is C\'ea's lemma. After that, it suffices to set $v_{h,M} = \Pi_x \otimes \Pi_y \fraku_\calY$, add and subtract $\Pi_x \fraku_\calY$, and use the stability of $\Pi_x$.
\end{proof}

\begin{remark}[regularity and rate of approximation]
  The derivation of error estimates for the solution of \eqref{eq:ExtensionWeakForm} requires estimating the interpolation errors described in \Cref{thm:DirectionalCea}. To do this, one uses the explicit representation  \eqref{eq:SolutionRepresentation} of the solution as it has been done in \cite{MR3348172,MR3989717}. This process is further complicated by the fact that fourth-order equations have a limited regularity shift. For instance, \cite[Corollary 7.3.2.5]{MR3396210} shows that, if $d=2$, $\Omega$ is a convex polygon, and $\LAP^2 \Phi \in L^2(\Omega)$, then $\Phi \in H^3(\Omega) \cap H^2_0(\Omega)$. This, however, does not imply that $u\in H^4(\Omega)$.
\end{remark}

\begin{remark}[nonconforming methods]
  The implementation of $H^2$-conforming methods in general domains is a complicated endeavor, and not every finite element library offers ready-made elements for this purpose. One possible solution is to use nonconforming methods such as discrete Kirchhoff triangles \cite[\S8.2.1]{MR3309171}, virtual elements \cite{MR4855619}, interior penalty methods \cite{MR431742,MR2142191}, or other approaches. Obtaining analogues to \Cref{thm:DirectionalCea} for these methods should follow standard techniques and encounter the usual complications associated with fourth--order problems. One must simply take advantage of the Cartesian product structure of the domain and the problem.
\end{remark}


%
%
%

\section*{Funding}

EO: ANID grant FONDECYT-1220156 and by USM through USM project 2025 PI LII 25 12.
AJS: NSF grant DMS-2409918.

\bibliographystyle{plain}
\bibliography{biblio}

@article {MusinaNazarov,
    AUTHOR = {Musina, R. and Nazarov, A.I.},
     TITLE = {Fractional operators as traces of operator-valued curves},
   JOURNAL = {J. Funct. Anal.},
  FJOURNAL = {Journal of Functional Analysis},
    VOLUME = {287},
      YEAR = {2024},
    NUMBER = {2},
     PAGES = {Paper No. 110443, 33},
      ISSN = {0022-1236,1096-0783},
   MRCLASS = {34G10 (35R11 46E40)},
  MRNUMBER = {4735195},
MRREVIEWER = {Vincenzo\ Ambrosio},
       DOI = {10.1016/j.jfa.2024.110443},
}

@article {MR4429579,
    AUTHOR = {Cora, G. and Musina, R.},
     TITLE = {The {$s$}-polyharmonic extension problem and higher-order
              fractional {L}aplacians},
   JOURNAL = {J. Funct. Anal.},
  FJOURNAL = {Journal of Functional Analysis},
    VOLUME = {283},
      YEAR = {2022},
    NUMBER = {5},
     PAGES = {Paper No. 109555, 33},
      ISSN = {0022-1236,1096-0783},
   MRCLASS = {26A33 (35J70 35R11 46E35)},
  MRNUMBER = {4429579},
MRREVIEWER = {Zhuoran\ Du},
       DOI = {10.1016/j.jfa.2022.109555},
}

@article {MR3773814,
    AUTHOR = {Chen, Y.K. and Lei, Z. and Wei, C.H.},
     TITLE = {Extension problems related to the higher order fractional {L}aplacian},
   JOURNAL = {Acta Math. Sin. (Engl. Ser.)},
  FJOURNAL = {Acta Mathematica Sinica (English Series)},
    VOLUME = {34},
      YEAR = {2018},
    NUMBER = {4},
     PAGES = {655--661},
      ISSN = {1439-8516,1439-7617},
   MRCLASS = {35R11 (35P30 42B37)},
  MRNUMBER = {3773814},
MRREVIEWER = {Phi\ Long\ Le},
       DOI = {10.1007/s10114-017-7325-6},
}

@article {MR4742773,
    AUTHOR = {Biswas, A. and Stinga, P.R.},
     TITLE = {Sharp extension problem characterizations for higher
              fractional power operators in {B}anach spaces},
   JOURNAL = {J. Funct. Anal.},
  FJOURNAL = {Journal of Functional Analysis},
    VOLUME = {287},
      YEAR = {2024},
    NUMBER = {3},
     PAGES = {Paper No. 110474, 27},
      ISSN = {0022-1236,1096-0783},
   MRCLASS = {35R11 (26A33 35A02 35C15 35S05 47D06)},
  MRNUMBER = {4742773},
MRREVIEWER = {Christopher\ Steven\ Goodrich},
       DOI = {10.1016/j.jfa.2024.110474},
}

@misc{1302.4413,
Author = {Yang, R.},
Title = {On higher order extensions for the fractional {L}aplacian},
Year = {2013},
note = {arXiv:1302.4413},
}

@article {MR247668,
    AUTHOR = {Mol\v{c}anov, S.A. and Ostrovski\u{i}, E.},
     TITLE = {Symmetric stable processes as traces of degenerate diffusion processes},
   JOURNAL = {Teor. Verojatnost. i Primenen.},
  FJOURNAL = {Akademija Nauk SSSR. Teorija Verojatnoste\u{i}\ i ee Primenenija},
    VOLUME = {14},
      YEAR = {1969},
     PAGES = {127--130},
      ISSN = {0040-361x},
   MRCLASS = {60.62},
  MRNUMBER = {247668},
MRREVIEWER = {J.\ G.\ Wendel},
}

@article {MR2354493,
    AUTHOR = {Caffarelli, L. and Silvestre, L.},
     TITLE = {An extension problem related to the fractional {L}aplacian},
   JOURNAL = {Comm. Partial Differential Equations},
  FJOURNAL = {Communications in Partial Differential Equations},
    VOLUME = {32},
      YEAR = {2007},
    NUMBER = {7-9},
     PAGES = {1245--1260},
      ISSN = {0360-5302,1532-4133},
   MRCLASS = {35J70},
  MRNUMBER = {2354493},
MRREVIEWER = {Francesco\ Petitta},
       DOI = {10.1080/03605300600987306},
}

@article {MR3989717,
    AUTHOR = {Banjai, L. and Melenk, J.M. and Nochetto, R.H. and Ot\'arola, E. and Salgado, A.J. and Schwab, C.},
     TITLE = {Tensor {FEM} for spectral fractional diffusion},
   JOURNAL = {Found. Comput. Math.},
  FJOURNAL = {Foundations of Computational Mathematics. The Journal of the
              Society for the Foundations of Computational Mathematics},
    VOLUME = {19},
      YEAR = {2019},
    NUMBER = {4},
     PAGES = {901--962},
      ISSN = {1615-3375,1615-3383},
   MRCLASS = {65N30 (26A33 35R11 65N12)},
  MRNUMBER = {3989717},
       DOI = {10.1007/s10208-018-9402-3},
}

@article {MR3348172,
    AUTHOR = {Nochetto, R.H. and Ot\'arola, E. and Salgado, A.J.},
     TITLE = {A {PDE} approach to fractional diffusion in general domains: a
              priori error analysis},
   JOURNAL = {Found. Comput. Math.},
  FJOURNAL = {Foundations of Computational Mathematics. The Journal of the
              Society for the Foundations of Computational Mathematics},
    VOLUME = {15},
      YEAR = {2015},
    NUMBER = {3},
     PAGES = {733--791},
      ISSN = {1615-3375,1615-3383},
   MRCLASS = {65N30 (35S11 65N12)},
  MRNUMBER = {3348172},
MRREVIEWER = {Patrick\ Henning},
       DOI = {10.1007/s10208-014-9208-x},
}

@misc{NIST:DLMF,
         key = "{\relax DLMF}",
       title = "{\it NIST Digital Library of Mathematical Functions}",
howpublished = "\url{https://dlmf.nist.gov/}, Release 1.2.4 of 2025-03-15",
         url = "https://dlmf.nist.gov/",
        note = "F.~W.~J. Olver, A.~B. {Olde Daalhuis}, D.~W. Lozier, B.~I. Schneider,
                R.~F. Boisvert, C.~W. Clark, B.~R. Miller, B.~V. Saunders,
                H.~S. Cohl, and M.~A. McClain, eds."
}

@book {MR2373954,
    AUTHOR = {Brenner, S.C. and Scott, L.R.},
     TITLE = {The mathematical theory of finite element methods},
    SERIES = {Texts in Applied Mathematics},
    VOLUME = {15},
   EDITION = {Third},
 PUBLISHER = {Springer, New York},
      YEAR = {2008},
     PAGES = {xviii+397},
      ISBN = {978-0-387-75933-3},
   MRCLASS = {65-01 (65-02)},
  MRNUMBER = {2373954},
       DOI = {10.1007/978-0-387-75934-0},
       URL = {https://doi.org/10.1007/978-0-387-75934-0},
}

@article {MR1266026,
    AUTHOR = {Laghchim-Lahlou, M. and Sablonni\`ere, P.},
     TITLE = {Triangular finite elements of {HCT} type and class {$C^\rho$}},
   JOURNAL = {Adv. Comput. Math.},
  FJOURNAL = {Advances in Computational Mathematics},
    VOLUME = {2},
      YEAR = {1994},
    NUMBER = {1},
     PAGES = {101--122},
      ISSN = {1019-7168,1572-9044},
   MRCLASS = {65D07 (65N30)},
  MRNUMBER = {1266026},
MRREVIEWER = {Gheorghe\ Micula},
       DOI = {10.1007/BF02519038},
       URL = {https://doi.org/10.1007/BF02519038},
}

@article {MR3163246,
    AUTHOR = {Walkington, N.J.},
     TITLE = {A {$C^1$} tetrahedral finite element without edge degrees of
              freedom},
   JOURNAL = {SIAM J. Numer. Anal.},
  FJOURNAL = {SIAM Journal on Numerical Analysis},
    VOLUME = {52},
      YEAR = {2014},
    NUMBER = {1},
     PAGES = {330--342},
      ISSN = {0036-1429,1095-7170},
   MRCLASS = {65N30},
  MRNUMBER = {3163246},
MRREVIEWER = {Gerard\ Awanou},
       DOI = {10.1137/130912013},
       URL = {https://doi.org/10.1137/130912013},
}

@incollection {MR1115237,
    AUTHOR = {Ciarlet, P.G.},
     TITLE = {Basic error estimates for elliptic problems},
 BOOKTITLE = {Handbook of numerical analysis, {V}ol.\ {II}},
    SERIES = {Handb. Numer. Anal.},
    VOLUME = {II},
     PAGES = {17--351},
 PUBLISHER = {North-Holland, Amsterdam},
      YEAR = {1991},
      ISBN = {0-444-70365-9},
   MRCLASS = {65-02 (65N15 65N30)},
  MRNUMBER = {1115237},
}

@article {MR431742,
    AUTHOR = {Baker, G.A.},
     TITLE = {Finite element methods for elliptic equations using
              nonconforming elements},
   JOURNAL = {Math. Comp.},
  FJOURNAL = {Mathematics of Computation},
    VOLUME = {31},
      YEAR = {1977},
    NUMBER = {137},
     PAGES = {45--59},
      ISSN = {0025-5718,1088-6842},
   MRCLASS = {65N30},
  MRNUMBER = {431742},
MRREVIEWER = {Reinhard\ Scholz},
       DOI = {10.2307/2005779},
       URL = {https://doi.org/10.2307/2005779},
}

@book {MR3309171,
    AUTHOR = {Bartels, S.},
     TITLE = {Numerical methods for nonlinear partial differential
              equations},
    SERIES = {Springer Series in Computational Mathematics},
    VOLUME = {47},
 PUBLISHER = {Springer, Cham},
      YEAR = {2015},
     PAGES = {x+393},
      ISBN = {978-3-319-13796-4; 978-3-319-13797-1},
   MRCLASS = {65-01 (35A15 35A35 65Mxx 65Nxx)},
  MRNUMBER = {3309171},
MRREVIEWER = {Karsten\ Urban},
       DOI = {10.1007/978-3-319-13797-1},
       URL = {https://doi.org/10.1007/978-3-319-13797-1},
}

@article {MR4855619,
    AUTHOR = {Bonnet, G. and Cangiani, A. and Nochetto, R.H.},
     TITLE = {Conforming virtual element method for nondivergence form
              linear elliptic equations with {C}ordes coefficients},
   JOURNAL = {Math. Models Methods Appl. Sci.},
  FJOURNAL = {Mathematical Models and Methods in Applied Sciences},
    VOLUME = {35},
      YEAR = {2025},
    NUMBER = {1},
     PAGES = {75--112},
      ISSN = {0218-2025,1793-6314},
   MRCLASS = {65N30 (65N12)},
  MRNUMBER = {4855619},
       DOI = {10.1142/S0218202525500034},
       URL = {https://doi.org/10.1142/S0218202525500034},
}

@article {MR2142191,
    AUTHOR = {Brenner, S.C. and Sung, L.-Y.},
     TITLE = {{$C^0$} interior penalty methods for fourth order elliptic
              boundary value problems on polygonal domains},
   JOURNAL = {J. Sci. Comput.},
  FJOURNAL = {Journal of Scientific Computing},
    VOLUME = {22/23},
      YEAR = {2005},
     PAGES = {83--118},
      ISSN = {0885-7474,1573-7691},
   MRCLASS = {65N30},
  MRNUMBER = {2142191},
       DOI = {10.1007/s10915-004-4135-7},
       URL = {https://doi.org/10.1007/s10915-004-4135-7},
}

@book {MR3396210,
    AUTHOR = {Grisvard, P.},
     TITLE = {Elliptic problems in nonsmooth domains},
    SERIES = {Classics in Applied Mathematics},
    VOLUME = {69},
      NOTE = {Reprint of the 1985 original [MR0775683],
              With a foreword by Susanne C. Brenner},
 PUBLISHER = {Society for Industrial and Applied Mathematics (SIAM),
              Philadelphia, PA},
      YEAR = {2011},
     PAGES = {xx+410},
      ISBN = {978-1-611972-02-3},
   MRCLASS = {35J25 (01A75 35-02)},
  MRNUMBER = {3396210},
       DOI = {10.1137/1.9781611972030.ch1},
       URL = {https://doi.org/10.1137/1.9781611972030.ch1},
}

\end{document}